\documentclass[12pt,a4paper]{article}
\usepackage[utf8]{inputenc}
\usepackage[english]{babel}
\usepackage{amsmath,amsthm,amsfonts,amssymb,amsrefs}
\usepackage{enumitem}
\usepackage{makeidx}
\usepackage{hyperref}

\usepackage{tikz}
\usetikzlibrary{arrows}
\usepackage{graphicx}
\usepackage[left=3cm,right=3cm,top=2cm,bottom=2cm]{geometry}

\newcommand{\N}{\mathbb{N}}
\newcommand{\C}{\mathbb{C}}
\newcommand{\CC}[1]{\mathbb{C}^{#1}}

\newcommand{\SpC}{\mathrm{Sp}_{2n}(\mathbb{C})}
\newcommand{\Eu}[1]{\begin{pmatrix}
I_n&#1\\0&I_n
\end{pmatrix}}
\newcommand{\El}[1]{\begin{pmatrix}
I_n&0\\#1&I_n
\end{pmatrix}}

\newcommand{\ringofholo}{\mathcal{O}(X)}
\newcommand{\Symplecto}{\mathrm{Aut}_{Sp}(\C^{2n})}
\newcommand{\SpanD}{\mathrm{span}\{\Delta\}}
\newcommand{\Norm}[1]{\vert\vert #1 \vert\vert}
\DeclareMathOperator{\img}{Im}

\newtheorem{theorem}{Theorem}[section]
\newtheorem{lemma}[theorem]{Lemma}
\newtheorem{cor}[theorem]{Corollary}
\newtheorem{prop}[theorem]{Proposition}
\newtheorem{defi}[theorem]{Definition}
\newtheorem{Rem}[theorem]{Remark}
\newtheorem{Exa}[theorem]{Example}
\newtheorem{question}[theorem]{Question}
\newtheorem{thm}{Theorem}

\author{Rafael B.\ Andrist$^1$, Gaofeng Huang$^2$, \\ Frank Kutzschebauch$^2$ and Josua Schott$^2$}
\title{Parametric Symplectic Jet Interpolation}
\date{$^1$University of Ljubljana \\ $^2$University of Bern}

\begin{document}
\reversemarginpar
\setlength{\marginparwidth}{3 cm}
\maketitle
\begin{abstract}
We prove a parametric jet interpolation theorem for symplectic holomorphic automorphisms of $\mathbb{C}^{2n}$ with parameters in a Stein space. Moreover, we provide an example of an unavoidable set for symplectic holomorphic maps.
\end{abstract}

\section{Introduction}

Since the late 1980s the group of
holomorphic automorphisms of $\C^n, n\ge 2,$ has been studied intensively.
It was long well known that this group is enormously big. It acts infinitely transitively, which by definition means that it acts transitively on finite subsets of any cardinality.
Even finite jet-interpolation by holomorphic automorphisms at finitely many points has been proved 1999 by Forstneri\v{c} \cite{MR1760722}.  

However, the group of holomorphic automorphisms does 
not act transitively on infinite discrete subsets, as found out by Rosay and Rudin in \cite{RosayRudin}. They called the 
subsets in the orbit of the set
$N := \{ (i, 0, \ldots, 0), i \in \N \} \subset \C^n $ tame subsets, and they showed that not all infinite discrete subsets are tame. Furthermore, Buzzard and Forstneri\v{c} showed in \cite{MR1758584} that finite jet-interpolation by holomorphic automorphisms can be done simultaneously at all points of a tame discrete set in $\C^n$.
Finally, Ugolini \cite{UgoJetInterpolation} showed a parametric version of this result:
If the finite jets depend holomorphically on a Stein parameter $x \in X$,
then (under the topological condition that their linear parts are null-homotopic) the interpolating holomorphic automorphism can be chosen holomorphically depending on $x \in X$.

In contrast to the group of holomorphic automorphisms of $\C^n$, the group of symplectic holomorphic automorphisms of $\C^{2n}$ has been studied much less. Forstneri\v{c} \cite{Forstneric:Actions} proved that this group is also enormously big. In fact, he proved the so-called symplectic density property for $\C^{2n}$ (for more details see \cite{FoK:thefirst}*{Section 2.5}). In the present paper we prove the analog of Ugolini's parametric jet interpolation result for symplectic holomorphic automorphisms of $\C^{2n}$.

Let us introduce the symplectic holomorphic setting:
Denoting by  $(z_1,\dots,z_{2n})$ the coordinates of $\C^{2n}$, we call a holomorphic automorphism $F$ of $\C^{2n}$ a \textit{symplectic automorphism} if it preserves the symplectic form
\[ \omega = \sum_{i=1}^n dz_i\wedge dz_{n+i},\]
that is, $F^*\omega=\omega$, where $F^*$ denotes the pullback by $F$. We let $\Symplecto$ denote the group of symplectic automorphisms of $\C^{2n}$. 

A discrete sequence of points $\{a_j\}_{j\in \mathbb{N}}\subset \C^{2n}$ (i.e.\ without limit points in $\C^{2n}$) without repetition is called \textit{symplectically tame} if there exists a symplectic automorphism $F \in \Symplecto$ such that 
\[ F(a_j)=j\Delta\quad \text{for all } j\in \mathbb{N},\]
where $\Delta=(1,\dots,1)^t\in \C^{2n}$. 

Let $k$ be a natural number. A mapping $F \colon \C^{2n}\to \C^{2n}$ is called \textit{symplectic of order $k$ at a point $p\in \C^{2n}$} if
$F^*\omega-\omega = \sum_{i<j} g_{ij}\ dz_i\wedge dz_j$ with $g_{ij}(z)=O(|z-p|^k)$ for $z\to p$ and $1\leq i<j\leq n$. We extend this property to the parametric setting: Consider a Stein space $X$. When we mention Stein spaces in this paper, we always assume that they are finite-dimensional and reduced. A holomorphic mapping $F \colon X\times \C^{2n}\to \C^{2n}$ is called $X$-symplectic of order $k$ at $p$, if there exist holomorphic functions $g_{ij} \colon X\times \C^{2n}\to \C$, $1\leq i<j\leq n$, such that
\[ (F_x)^*\omega - \omega = \sum_{i<j} (g_{ij})_x dz_i\wedge dz_j,\quad x\in X,\]
and $(g_{ij})_x(z)=O(|z-p|^k)$ for every $x\in X$.

Let $p\in \C^{2n}$, $k\in \mathbb{N}$ and $f,g \colon \C^{2n}\to \C^{2n}$ be two holomorphic maps. We declare $f$ and $g$ to be equivalent, if they have the same Taylor polynomial of degree $k$ at $p$, that is, if
\[ f(z)-g(z) = O(|z-p|^{k+1}),\quad z\to p.\]
This induces an equivalence relation on the space of holomorphic maps $\C^{2n}\to\C^{2n}$. The equivalence classes are called $k$-jets at $p$. Let $J^k_{p,*}(\C^{2n})$ denote the space of all $k$-jets at $p$. By abuse of notation, we don't distinguish between the representing Taylor polynomial and the corresponding $k$-jet.  
For $q\in \C^{2n}$, let 
\[J^k_{p,q}(\C^{2n}) = \{ P\in  J^k_{p,*}(\C^{2n}): P(p)=q\}. \] 


The main result of our paper is a full analog of Ugolini's
Theorem 1.1. in \cite{UgoJetInterpolation} in the symplectic holomorphic setting. We first recall the following definition.

\begin{defi}
Let $X$ be a complex space. A map $F \colon X\to \Symplecto$ is called holomorphic if the evaluation map $F(x)(z)=:F_x(z)$ is holomorphic in the usual sense as a map $X\times \C^{2n}\to \C^{2n}$. 
\end{defi}

\begin{thm}\label{theorem1}
Let $X$ be a Stein space, $\{a_j\}_{j\in \mathbb{N}},\{b_j\}_{j\in \mathbb{N}}\subset \C^{2n}, \, n \in \mathbb{N},$ be symplectically tame sequences of points and $m_j\in \mathbb{N}_{+}, j>0$. For every $j\in \mathbb{N}$, let $P^j \colon X\to J_{a_j,b_j}^{m_j}(\C^{2n})$ be a holomorphic family of $m_j$-jets such that $P_x^j(a_j)=b_j$ for all $x\in X$. Assume that $P^j$ is $X$-symplectic of order $j$ at $a_j$ for every $j\in \mathbb{N}$. Then there exists a null-homotopic holomorphic $F \colon X\to \Symplecto$ such that
\[ F_x(z) = P_x^j(z) + O\left(\vert z-a_j\vert^{m_j+1}\right) \text{ for }z\to a_j, \ j\in \mathbb{N},\ x\in X    \]
if and only if the linear part map $Q^j \colon X\to \SpC$ of $P^j$ at $a_j$ is nullhomotopic for every $j\in \mathbb{N}$. 
\end{thm}

A non-parametric version of our result at a single point instead of a tame subset has been proved by L\o w, Pereira, Peters and Wold in \cite{LowEtAl}. The proof of our Theorem uses a deep result (see Theorem \ref{Vaserstein}) proved for general $n$ by the fourth author \cite{Schott}
after it was proved for $n=2$ by Ivarsson, L\o w and the third author.

The paper is organized as follows. In Section 2 we
recall some known facts, in particular Theorem \ref{Vaserstein} by the fourth author, and prove some auxiliary results. Section 3 contains the proof of our interpolation result for finitely many points. The induction process which leads to the full proof of the main result is contained in Section 4. In Section 5 we give some examples of symplectically tame sets and formulate some open questions related to this notion. We also construct an example of an unavoidable set for symplectic holomorphic maps that fix the origin. 

%
%
\section{Preliminaries and Notations}
Let $(z_1,\dots,z_{2n})\in\C^{2n}$ denote the coordinates of the complex Euclidean vector space of dimension $2n$, and let $e_1,\dots,e_{2n}$ be the standard basis vectors. We equip $\C^{2n}$ with the standard symplectic form
\[ \omega = \sum_{i=1}^n dz_i\wedge dz_{n+i}.\]
A holomorphic automorphism $F$ of $\C^{2n}$ is called a \textit{symplectic automorphism} if it preserves the symplectic form, i.e.\   $F^*\omega=\omega$, where $F^*$ denotes the pullback by $F$. We let $\Symplecto$ denote the group of (holomorphic) symplectic automorphisms on $\C^{2n}$. It was shown by Forstneri\v{c} \cite{Forstneric:Actions}*{Theorem 5.1} that $\Symplecto$ contains a dense subgroup generated by shears of the form 
\begin{align}\label{equation:symplectomorph}
 F(z):=z+f(z^tJv)v,
\end{align}
where $v\in \C^{2n}$, $f \colon \C\to\C$ is an entire function, $z^t$ denotes the transpose of $z$ and $J$ is the block matrix
\[ J = \begin{pmatrix}
0 & I_n\\ -I_n & 0
\end{pmatrix}\]
where $I_n$ is the $n\times n$ identity matrix and $0$ the $n\times n$ zero matrix. We introduce the following notation. For $v\in \C^{2n}$ let $\lambda_v(z):=z^tJv$. Then (\ref{equation:symplectomorph}) can be written as $F(z)=z+f(\lambda_v(z))v$.

%
%

We follow the convention of Rosay and Rudin \cite{RosayRudin} and call a set $E \subset \C^{2n}$ \emph{discrete} if $E$ has no limit points in $\C^{2n}$.

\begin{defi}
A discrete sequence of points $\{a_j\}_{j\in \mathbb{N}}\subset \C^{2n}$ without repetition is \textit{symplectically tame} if there exists a symplectic automorphism $F\in \Symplecto$ such that 
\[ F(a_j)=j\Delta\quad \text{for all } j\in \mathbb{N},\]
where $\Delta:=(1,\dots,1)^t\in \C^{2n}$. 
\end{defi}

\begin{Rem}
This definition is equivalent to the one given in \cite{AndristUgolini}*{Def.\ 3.2}, where $\Delta$ is replaced by the unit vector $e_1$. To see this, note that the symplectic automorphism
\[ \Psi(z):= z + \lambda_{v}(z) v,\quad v:=\Delta-e_1,\]
maps $\alpha e_1$ to $\alpha\Delta$ for every $\alpha\in \C$.
\end{Rem}

The following proposition is an extension of \cite{AndristUgolini}*{Lemma 3.3} (a symplectic analog of \cite{RosayRudin}*{Proposition 3.1}) which says that any two discrete sequences in $\SpanD$ can be permuted by an automorphism $F\in \Symplecto$. In addition, we require that $F$ agrees with a translation up to a given order $m_j$ at every $j\Delta$.

\begin{prop}\label{prop:symptame}
Let $\{c_j\}_{j\in \mathbb{N}}\subset \SpanD$ be a discrete sequence (without repetition) and  ${\{m_j\}_{j\in \mathbb{N}}\subset \mathbb{N}}$. Then there exists an automorphism $F\in \Symplecto$ such that
\[ F(z)=c_j+(z-j\Delta) + O(|z-j\Delta|^{m_j+1}),\quad z\to j\Delta, j\in \mathbb{N}.\]
\end{prop}

\begin{proof}
For each $j\in \mathbb{N}$ there exists $\gamma_j\in \C$ such that $c_j=\gamma_j\Delta$, by assumption. Let us write $\tilde{\Delta}:=J\Delta$. Observe that $\lambda_{\tilde{\Delta}}(j\Delta)=2nj$, $j\in \mathbb{N}$, is a discrete sequence in $\C$. By Mittag-Leffler's osculation theorem, there exists a holomorphic function $f_1 \colon \C \to \C$ such that
\[f_1(\zeta)=j+O(|\zeta - 2nj|^{m_j+1}),\quad \zeta \to 2nj,\quad j\in \mathbb{N}.\]
The symplectic automorphism $\Psi_1(z):=z+f_1(\lambda_{\tilde{\Delta}}(z))\tilde{\Delta}$ satisfies
\[\Psi_1(z)=j(\Delta+\tilde{\Delta}) + (z-j\Delta) + O(|z-j\Delta|^{m_j+1}),\quad z\to j\Delta,\]
that is, it maps $j\Delta$ to $j(\Delta+\tilde{\Delta})$ and agrees with a translation to order $m_j+1$ at $j\Delta$.

Similarly, we note that $\lambda_{\Delta}(j(\Delta+\tilde{\Delta}))=2nj$, $j\in \mathbb{N}$, is again a discrete sequence in $\C$. With the same argument as in the previous step, there exists a holomorphic function $f_2$ such that \[f_2(\zeta)=(\gamma_j-j) + O(|\zeta-2nj|^{m_j+1}),\quad\zeta \to 2nj,\quad j>0.\] Hence the symplectic automorphism $\Psi_2(z):=z+f_2(\lambda_{\Delta}(z))\Delta$ maps $j(\Delta+\tilde{\Delta})$ to $\gamma_j\Delta+j\tilde{\Delta}$ and agrees with a translation to order $m_j+1$ at $j(\Delta+\tilde{\Delta})$. 

Once more, we have a discrete sequence $\lambda_{\tilde{\Delta}}(\gamma_j\Delta + j\tilde{\Delta})=-2n \gamma_j$, $j\in \mathbb{N}$, in $\C$. Another application of Mittag-Leffler's osculation theorem implies the existence of a holomorphic function $f_3$ such that \[f_3(\zeta)=j+O(|\zeta  +2n\gamma_j|^{m_j+1}),\quad \zeta \to  -2n\gamma_j,\quad j>0.\]
The symplectic automorphism $\Psi_3(z):=z+f_3(\lambda_{\tilde{\Delta}}(z))\tilde{\Delta}$ maps $\gamma_j\Delta + j\tilde{\Delta}$ to $\gamma_j\Delta$ and agrees to a translation to order $m_j+1$ at $\gamma_j\Delta + j\tilde{\Delta}$.
Then, $\Psi:=\Psi_3\circ \Psi_2\circ \Psi_1$ is the desired symplectic automorphism, which maps $j\Delta$ to $c_j$ and agrees with a translation to order $m_j+1$ at $j\Delta$.
\end{proof}

Recall that a holomorphic mapping $F \colon X\times\C^{2n}\to \C^{2n}$ is $X$-symplectic of order $k$, $k\in \mathbb{N}$, at a point $p\in \C^{2n}$ if there exist holomorphic functions $g_{ij} \colon X\times \C^{2n}\to \C$, $1\leq i<j\leq n$, such that
\[ (F_x)^*\omega - \omega = \sum_{i<j} (g_{ij})_x dz_i\wedge dz_j,\quad x\in X,\]
and $(g_{ij})_x(z)=O(|z-p|^k)$ for every $x\in X$.
\begin{lemma}
Let $X$ be a Stein space, $k\in \mathbb{N}$, $p\in \C^{2n}$ and $F \colon X \times \C^{2n}\to \C^{2n}$ a holomorphic mapping,
\[ F_x(z) = F_x(p) + L_x\cdot (z-p) + O(|z-p|^2),\quad z\to p,\quad  x\in X,\]
where $L_x$ is some $2n\times 2n$-matrix depending holomorphically on $x$. Assume that $F$ is $X$-symplectic of order $k$ at $p\in \C^{2n}$. Then the linear part map $Q_x=L_x\cdot (z-p) $ of $F_x$ at $p$ is symplectic for every $x\in X$, and in particular, $L_x$ gives rise to a holomorphic mapping $L \colon X\to \SpC$.
\end{lemma}
\begin{proof}
A calculation in coordinates shows
\begin{equation}\label{equation:symplecticOforder}
F_x^*\omega - \omega = Q_x^*\omega-\omega + \sum_{i<j}(\tilde{g}_{ij})_x\  dz_i\wedge dz_j,
\end{equation}
where $\tilde{g}_{ij} \colon X\times \C^{2n}\to \C$ are some holomorphic functions with $(\tilde{g}_{ij})_x(z)=O(|z-p|)$, $z\to p$, for every $x\in X$. Moreover, the coefficients of $Q_x^*\omega-\omega$ are constant in $z$, as the matrix $L_x$ is constant in $z$.

Furthermore, the left-hand side in (\ref{equation:symplecticOforder}) equals $\sum_{i<j} (g_{ij})_x\ dz_i\wedge dz_j$, for some holomorphic functions $g_{ij} \colon X\times \C^{2n}\to \C$ with $(g_{ij})_x(z)=O(|z-p|^k)$, for every $x\in X$, by assumption. Choose $z=p$ to conclude $Q_x^*\omega-\omega \equiv 0$. Thus $Q_x$ is symplectic and, in particular, $L_x$ is a symplectic matrix for every $x\in X$.
\end{proof}

\begin{prop}\label{prop:Hamiltonian}
Let $X$ be a Stein space, $k\in \mathbb{N}$, $\hat{N}:= \genfrac(){0pt}{0}{2n+k}{2n-1}$ and let $P \colon X\times \C^{2n}\to \C^{2n}$ be a holomorphic mapping such that
\[ P_x(z) = z + P_x^k(z) + O(|z|^{k+1}),\quad z\to 0,\quad x\in X,\]
where $P_x^k$ is an $k$-homogeneous polynomial mapping on $\C^{2n}$. Assume that $P$ is $X$-symplectic of order $k$ at the origin. Then there exist $b_1,\dots,b_{\hat{N}}\in \C^{2n}$ and holomorphic functions $c_1,\dots,c_{\hat{N}} \colon X \to \C$ such that
\begin{align}\label{equationP}
    P_x^k(z) = \sum_{j=1}^{\hat{N}} c_j(x) \cdot (b_j^tJz)^k\cdot b_j,\quad \forall x\in X.
\end{align}
\end{prop}

\begin{proof}
Observe that $P_x^k$ can be interpreted as a polynomial vector field on $\C^{2n}$. According to \cite{LowEtAl}*{Lemma 3.1}, $P_x^k$ is a symplectic vector field for every $x\in X$, i.e.\ $d(\iota_{P_x^k}\omega)=0$, since $P_x$ is symplectic of order $k$ at the origin. Since $H^1(\C^{2n},\C)=0$, there exists a holomorphic (even polynomial) Hamiltonian $H_x \colon \C^{2n}\to \C$ such that $P_x^k = J\cdot DH_x$, where $DH_x = (\frac{\partial H_x}{\partial z_1},\dots, \frac{\partial H_x}{\partial z_{2n}})^t$. $H_x$ is unique up to a constant term. We can therefore choose $H_x$ such that $H_x(0)= 0$. This choice implies that $H_x$ is $(k+1)$-homogeneous. Moreover, it implies that $H(x,z):=H_x(z)$ is holomorphic as a map $X\times \C^{2n}\to \C$, since $P_x^k$ depends holomorphically on $x$. 
We can choose $b_1,\dots,b_{\hat{N}}\in \C^{2n}$ such that $\{(b_j^tz)^{k+1}: j=1,\dots,\hat{N}\}$ forms a basis of the vector space of $(k+1)$-homogeneous polynomials in $2n$ variables, by \cite{LowEtAl}*{Lemma 2.1}. Since $J$ is an isomorphism, $\{(b_j^tJz)^{k+1}: j=1,\dots,\hat{N}\}$ is also a basis.

Hence
\[ H_x(z) = \sum_{j=1}^{\hat{N}} \tilde{c}_{j,x}\cdot (b_j^tJz)^{k+1}\]
and the coefficients $\tilde{c}_{j,x}$ depend holomorphically on $x\in X$. Observe that
\[ D((b_j^tz)^{k+1}) = (k+1)(b_j^tz)^k\cdot b_j\]
and therefore
\[P_x^k(z) = J\cdot DH_x(z) = -\sum_{j=1}^{\hat{N}} (k+1)\tilde{c}_{j,x}\cdot (b_j^tJz)^k b_j,\]
since $J^2=-I_{2n}$. Defining
\[ c_j(x):= -(k+1)\tilde{c}_{j,x}\]
implies formula (\ref{equationP}).

The dimension of the vector space of $k$-homogeneous polynomial Hamiltonian vector fields on $\C^{2n}$ is given by
\[ \hat{N}:= \genfrac(){0pt}{0}{2n+k}{2n-1}\]
(see \cite{Lin}*{Remark 3.9}).
\end{proof}


The next result is \cite{UgoJetInterpolation}*{Lemma 2.6} and turns out to be very helpful.
\begin{lemma}\label{lemma:magicTool}
Let $T\subset X$ be a compact set and $K\subset \C$ a convex compact set such that $0\not\in K$. Let $\{a_i\}_{i=0}^m\subset \C\setminus \{0\}$ and $\{c_j\}_{j\in \mathbb{N}}\subset \C\setminus \{0\}$ be a discrete sequence. Given $\beta\in \ringofholo$, $\epsilon>0$ and $r,N\in \mathbb{N}$, there exists a holomorphic $f \colon X\times \C\to \C$ such that
\begin{enumerate}[label=(\roman*)]
\item $|f|_{T\times K}<\epsilon$
\item $f_x(\zeta) = \beta(x)\zeta^r + O(|\zeta|^{r+1})$ for $\zeta \to 0$
\item $f_x(\zeta) = O(|\zeta - a_i|^N)$ for $\zeta \to a_i$, $i=1,\dots,m$
\item $f_x(c_j)=0$, for $j\in \mathbb{N}$.
\end{enumerate}
\end{lemma}

In the following we discuss the factorization of symplectic matrices into elementary factors. For a symmetric $n\times n$-matrix $A$, i.e.\ $A^t=A$, the matrices
\begin{align}\label{elemSymp}
\Eu{A},\quad \El{A}
\end{align}
are symplectic. In \cite{Schott}, matrices of the form (\ref{elemSymp}) are called elementary symplectic matrices. However, for this paper we need even simpler matrices. We choose a basis $\{\tilde{E}_{ij}\}_{1\leq i\leq j\leq n}$ of the vector space of $n\times n$ symmetric matrices, where $\tilde{E}_{ij}$ is defined by
\[ \tilde{E}_{ij} := \begin{cases} E_{ij} + E_{ji} + E_{ii} + E_{jj} &i\neq j\\
E_{ii} & i=j\end{cases}, \quad 1\leq i\leq j \leq n\]
and $E_{ij}$ is the matrix having a one at entry $(i,j)$ and zeros elsewhere. In this article, we call matrices of the form
\begin{align}\label{elemSymp2}
 \Eu{\alpha \tilde{E}_{ij}},\quad \El{\alpha\tilde{E}_{ij}}\tag{$\star$}
\end{align}
\emph{elementary symplectic matrices}. Observe that each symplectic matrix of the form (\ref{elemSymp}) is a finite product of elementary symplectic matrices. This follows from the fact that
\[\Eu{A}\Eu{B} = \Eu{A+B}\]
for all symmetric matrices $A,B$.

\begin{theorem}[Symplectic Vaserstein Problem, \cite{Schott}*{Main Theorem}]\label{Vaserstein}Let $X$ be a finite dimensional Stein space. Then a holomorphic mapping $f \colon X\to \SpC$ factorizes into a finite product of elementary symplectic matrices if and only if $f$ is null-homotopic.
\end{theorem}


Now we want to establish a relationship between elementary symplectic matrices and certain symplectic shears. For $1\leq i\leq j\leq n$, we define 
\[ \tilde{e}_{ij}=(-1)\cdot\begin{cases} e_i+e_j & i\neq j\\ e_i & i=j\end{cases} \quad \text{and} \quad \tilde{f}_{ij}=\begin{cases} e_{n+i}+e_{n+j} & i\neq j\\ e_{n+i} & i=j\end{cases}.\]
The factor $(-1)$ in the definition of $\tilde{e}_{ij}$ was chosen for purely technical reasons.
\begin{lemma}
Let $\alpha\in \C$ and $f \colon \C\to \C$ holomorphic with $f(\zeta) = \alpha\zeta + O(|\zeta|^2)$ for $\zeta \to 0$. Further, let $\Psi_1(z):=z+f(\lambda_{\tilde{e}_{ij}}(z))\tilde{e}_{ij}$ and $\Psi_2(z):=z+f(\lambda_{\tilde{f}_{ij}}(z))\tilde{f}_{ij}$. Then we have
\[ \Psi_k(z)= A_kz + O(|z|^2),\quad z\to 0,\ k=1,2\]
where 
\begin{align*}
A_1=\Eu{-\alpha\tilde{E}_{ij}},\quad A_2=\El{\alpha\tilde{E}_{ij}}.
\end{align*}
We call $\Psi_k$ a \emph{symplectic shear of $A_k$}.
\end{lemma}
\begin{proof}
This follows from the fact that $A_k=(D\Psi_k)_0$ and $f'(0)=\alpha$.
\end{proof}
\begin{Rem}\label{germRemark}
Observe that $\Psi_k$ remains a symplectic shear of $A_k$ if we replace $f$ by a holomorphic function $g$ with the same linear part at $0$, i.e.\ if $g(\zeta)=\alpha\zeta + O(|\zeta|^2)$, $\zeta \to 0$.
\end{Rem}

From the chain rule we deduce
\begin{cor}\label{corollary:factorizationLinearPart} Let $A_1,\dots, A_L$ be elementary symplectic matrices and $\Psi_j$ be a symplectic shear of $A_j$, $j=1,\dots,L$. Then
\[ \Psi_L\circ \cdots \circ \Psi_1(z) = A_L\cdots A_2A_1z + O(|z|^2),\quad z\to 0.\]
\end{cor}

\section{Interpolation at finitely many points}
The following statement is the heart of the proof of Theorem \ref{theorem1} and in this section we turn to its proof.

Recall that $\Delta=(1,1,\dots,1)^t\in \C^{2n}$.
\begin{prop}\label{theorem:finiteInterpol}
Let $X$ be a Stein space, $k\in \mathbb{N}$, $p,q\in \C^{2n}$ and $P \colon X\to J^k_{p,q}(\C^{2n})$ be a holomorphic family of $k$-jets at $p$ with $P_x(p)=q$ for every $x\in X$. Assume that $P_x$ is $X$-symplectic of order $k$ at $p$ and  that the linear part map $Q \colon X\to \SpC$ of $P$ at $p$ is null-homotopic.
Given finitely many points $a_1,\dots,a_m\subset \SpanD \setminus \{p,q\}$, a natural number $N$, $\epsilon >0$, a compact set $T\subset X$, and a compact convex set $K\subset \C^{2n}\setminus \{p,q\}$, there exists a holomorphic map $F \colon X\to \Symplecto$ satisfying the following conditions:
\begin{enumerate}[label=(\roman*)]
\item $F_x(z) = P_x(z) + O(|z-p|^{k+1})$ for $z\to p$ and for every $x\in X$.
\item $F_x(z) = z + O(|z-a_i|^N)$ for $z\to a_i$, $1\leq i\leq m$, and for every $x\in X$.
\item $|F_x(z)-z|<\epsilon$ for every $x\in T$ and $z\in K$.
\item If $\{c_j\}_{j\in \mathbb{N}} \subset \C^{2n}\setminus (K\cup \{p,q\})$ is a discrete sequence contained in $\SpanD$, then we can ensure that $F_x(c_j)=c_j$ for every $x\in X$ and $j\in \mathbb{N}$.
\end{enumerate}
\end{prop}

\subsection{Proof of Proposition \ref{theorem:finiteInterpol}}
Thanks to Lemma \ref{lemma:symplectomorphism} (see the end of the current section), it is enough to prove the proposition in the special case $p=q=0$.
The strategy of the proof follows the idea of \cite{UgoJetInterpolation}*{Proof of Proposition 2.5}. We shall inductively construct symplectic automorphisms $S_j^x \colon \C^{2n}\to \C^{2n}$ for $j=1,\dots,k$ depending holomorphically on $x\in X$ and satisfying the following properties for every $r\in \{1,\dots,k\}$:
\begin{enumerate}
    \item[($a_r$)] $P_x \circ (S_1^x)^{-1} \circ \cdots \circ (S_r^x)^{-1}(z) = z+O(|z|^{r+1})$ for $z \to 0$
    \item[($b_r$)] $S_r^x \circ \cdots \circ S_1^x(z) = z+O(|z-a_i|^N)$ for $z \to a_i$, $1 \leq i \leq m$
    \item[(iii')] $|S_j^x(z)-z| < \frac{\epsilon}{k+1}$, for every $x\in T$, $z\in K$ and $j=1,\dots, k$,
    \item[(iv')]  If $\{c_j\}_{j\in \mathbb{N}} \subset \C^{2n}\setminus (K\cup \{p,q\})$ is a discrete sequence contained in $\SpanD$, then we can ensure that $S_j^x(c_i)=c_i$ for every $x\in X$, $i\in \mathbb{N}$ and $j=1,\dots, k$.

\end{enumerate}
Taking 
\[F_x(z):=S_k^x\circ \cdots \circ S_1^x(z),\quad x\in X\]
will furnish a holomorphic map $F \colon X\to \Symplecto$ satisfying conditions (i)--(iv).

\subsubsection{Base case: linear part}
In this section, we construct the map $S_1^x$. Before we can do that, we need some useful terms. Observe that if we look at the linear part of the jet, we have
\[ P_x(z) =  Q_xz +  O(|z|^2),\quad z\to 0.\]
So we need $S_1^x(z) = Q_xz +  O(|z|^2)$ as $z \to 0$ for ($a_1$) to be satisfied. Since $Q \colon X\to \SpC$ is a null-homotopic map by assumption, it factorizes into a finite product of elementary factors by Theorem \ref{Vaserstein},
i.e.\ $Q_x = A_L \cdots A_2A_1$ where each factor $A_l$ is an elementary symplectic matrix of the form
\[  \Eu{ -\alpha_x\tilde{E}_{ij}}\quad \text{or} \quad \El{\alpha_x\tilde{E}_{ij}},\quad \]
for some holomorphic function $\alpha \colon X\to \C$ and some $1\leq i\leq j\leq n.$ For each $A_l$, $l=1,\dots,L$, choose a symplectic shear $\Psi_l$ of $A_l$. By Corollary \ref{corollary:factorizationLinearPart}, we have
\[ \Psi_L\circ \cdots \circ \Psi_1(z) = Q_xz + O(|z|^2),\quad z\to 0,\]
hence property ($a_1$) is satsified. Note that each $\Psi_i$, $i=1,\dots,L$, is of the form
\begin{align}\label{sympshear}\tag{$\star$}
\Psi_i(z) = z + f_{i}(x,\lambda_v(z))v
\end{align}
for some $v\in \C^{2n}$ and some holomorphic function $f_i \colon X\times \C\to \C$. Write $f_i^x(z):=f_i(x,z)$. By Remark \ref{germRemark}, property ($a_1$) remains satisfied, as long as we keep the linear part of $f_i^x$ at the origin. Since the origin is neither in $\{a_1,\dots,a_m\}$ nor in $K$ by assumption, we can adjust $f$ in small enough neighborhoods around $a_i$, $i=1,\dots,m$, and $K$ so that property ($a_1$) remains valid.

We want to apply Lemma \ref{lemma:magicTool} in order to impose more conditions on $f_i$. Thanks to the following lemma we will be able to do this.
\begin{lemma}\label{lemma:discrete}
Let $\{c_l\}_{l\in \mathbb{N}}$ be a discrete sequence contained in $\SpanD$ and let $1\leq i\leq j\leq n$. 
Then $\{c_l\}_{l\in \mathbb{N}}$ is mapped injectively by $\lambda_{\tilde{e}_{ij}} \colon \C^{2n}\to \C$ and $\lambda_{\tilde{f}_{ij}} \colon \C^{2n}\to \C$. Moreover, the images $\{\lambda_{\tilde{e}_{ij}}(c_l)\}_{l\in \mathbb{N}}$ and $\{\lambda_{\tilde{f}_{ij}}(c_l)\}_{l\in \mathbb{N}}$ are discrete.
\end{lemma}

\begin{proof}
Recall that $\lambda_{v}(z)=z^tJv$ for $v=\tilde{e}_{ij},\tilde{f}_{ij}$. Then we have \[\lambda_{\tilde{e}_{ij}}(\alpha \Delta)=\lambda_{\tilde{f}_{ij}}(\alpha\Delta)=  \begin{cases}
2\alpha  &i\neq j\\ \alpha &i=j \end{cases},\quad 1\leq i\leq j\leq n,\]
for all $\alpha\in \C$.
\end{proof}

We define
\[ K^\epsilon:=\{z\in \C^{2n}: \inf_{w\in K} |z-w| \leq \epsilon\}\]
which is convex and compact, since $K$ is convex and compact. By Lemma \ref{lemma:discrete}, the set $\{\lambda_v(c_j)\}_{j\in \mathbb{N}}\subset \C$ is discrete (note that $\lambda_v$ is the linear map from (\ref{sympshear})). According to Lemma \ref{lemma:magicTool}, $f_i$ can be chosen such that, in addition to property
\begin{align}\label{equation:linearpart}
f_i^x(\zeta) = \alpha_x\zeta + O(|\zeta|^2),\quad \zeta\to 0, \, x\in X,
\end{align}
also the following properties are satisfied:
\begin{align}
&|f_i|_{T\times \lambda_v(K^\epsilon)}<\frac{\epsilon}{|v|(k+1)L},\label{inequation:approx}\\
&f_i^{x}(\zeta)=O(|\zeta - \lambda_{v}(a_j)|^N), \quad \zeta\to \lambda_{v}(a_j),\quad 1\leq j\leq m,\quad x\in X,\label{zeros:order}\\
&f_i^{x}(\lambda_v(c_j))=0, \quad j\in \mathbb{N},\quad x\in X.\label{equation:discreteset}
\end{align}

Define $S_1^x := \Psi_L\circ \cdots \circ \Psi_1$. We already know that $S_1^x$ satisfies property ($a_1$). Properties ($b_1$) and (iv') are satisfied because of (\ref{zeros:order}) and (\ref{equation:discreteset}), respectively. Moreover, this choice of $S_1^x$ satisfies property (iii') i.e.\
\[  |S_1^x(z)-z|<\frac{\epsilon}{k+1},\quad x\in T, z\in K,
\]
by (\ref{inequation:approx}).
\subsubsection{Induction step}

For the induction step, we assume that for some $r \in \{2,\dots,k\}$ we have already found maps  $S_1,S_2,\dots, S_{r-1}$ such that conditions (iii'), (iv') and
\begin{enumerate}
    \item[($a_{r-1}$)] $P_x^j \circ (S_1^x)^{-1}\circ \cdots \circ (S_{r-1}^x)^{-1} = z + O(|z|^r)$ for $z \to 0$
    \item[($b_{r-1}$)] $S_{r-1}^x \circ \cdots \circ S_1^x(z) = z + O(|z-a_i|^N)$ for $z \to a_i, \, 1 \leq i \leq m$
\end{enumerate}
are satisfied. Then
\[ \tilde{P}_x:= P_x\circ (S_1^x)^{-1}\circ \cdots \circ (S_{r-1}^x)^{-1}(z) = z + P_x^r + O(|z|^{r+1}), \, z\to 0,\]
where $P_x^r$ is a homogeneous polynomial of order $r$ on $\C^{2n}$ depending holomorphically on $x\in X$. Moreover, we may interpret $P_x^r$ as a homogeneous vector field on $\C^{2n}$. Observe that $\tilde{P}$ is $X$-symplectic of order $r$ at the origin, since $P$ is $X$-symplectic of order $r$ at the origin and $(S_1^x)^{-1}\circ \cdots \circ (S_{r-1}^x)^{-1}$ is an automorphism fixing the origin. 

Proposition \ref{prop:Hamiltonian} implies the existence of vectors
$b_1,\dots,b_{\hat{N}}\in \C^{2n}$ and holomorphic functions $c_1,\dots,c_{\hat{N}} \colon X\to \C$ such that
\[ P_x^r(z) = \sum_{j=1}^{\hat{N}} c_j(x) (b_j^TJz)^r b_j.\]

Note that the vectors $b_1,\dots,b_{\hat{N}}$ can be perturbed slightly and still have the property that \[{(b_1^TJz)^r,\dots,(b_{\hat{N}}^TJz)^r}\] forms a basis of the vector space of $r$-homogeneous polynomials in $2n$ variables. We may therefore assume that the image of the discrete set
\[ \{0\}\cup \{a_1\}\cup \cdots \{a_m\}\cup \bigcup_{i\geq 1} \{c_i\}\]
is again discrete under the map $\lambda_{b_j}(z)^r$ and without repetition, since all points of that set are contained in $\SpanD$. To simplify notation, we write $\lambda_j$ instead of $\lambda_{b_j}$. By Lemma \ref{lemma:magicTool} there exist holomorphic functions $g_j \colon X\times \C\to \C$ with
\begin{enumerate} [label=\arabic*)]
    \item $|g_j|_{T\times \lambda_j(K^\epsilon)^r}< \epsilon / (\hat{N}(k+1))$
    \item For $\zeta \to 0$, $x\in X$ we have $g_{j}^x(\zeta)=c_x\zeta^r + O(|\zeta|^{r+1})$ 
    \item For $\zeta \to \lambda_j(a_i)$, $1 \leq i \leq m$, $x \in X$ we have $g_{j}^x(\zeta) = O(|\zeta-\lambda(a_i)|^N)$
    \item For $i>0$ we have $g_{j}^x(c_i)=c_i$ 
\end{enumerate}
Then we define the mappings
\[ G_j^x(z) = z + g_{j}^x(\lambda_{b_j}(z)^r)b_j,\quad j=1,\dots,\hat{N}.\]
Any composition of those $G_j^x$'s yields the desired symplectic automorphism $S_r^x$. And this finishes the induction step.

After finitely many steps we find symplectic automorphisms $S_0^x,\dots,S_k^x$ such that conditions ($a_k$) and ($b_k$) are satisfied. Then $F_x(z):=S_k^x\circ \cdots \circ S_0^x(z)$ furnishes a holomorphic map $F \colon X\to \Symplecto$ satisfying the required conditions. This proves the proposition for $p=q=0$.

In order to finish the proof, it remains to show
\begin{lemma}\label{lemma:symplectomorphism}
Let $p,q\in \C^{2n}, p\neq q,$ be two different points, $K\subset \C^{2n}\setminus \{p,q\}$ a convex, compact set and $\{a_j\}_{j=1}^m\subset \C^{2n}\setminus (K\cup \{p,q\})$ a finite set of points. Let $\{c_j\}_{j\in \mathbb{N}}\subset \SpanD\setminus \{p,q,a_1,\dots, a_m\}$ be a discrete set, $\epsilon>0$ and $N\in \mathbb{N}$. Then there exists a symplectic automorphism $F \colon \C^{2n}\to\C^{2n}$ such that
\begin{enumerate}[label=(\roman*)]
    \item $F(p)=q$
    \item $F(z)=z+O(|z-a_j|^N), \, z\to a_j$,
    \item $|F(z)-z|_K<\epsilon$,
    \item $F(c_j)=c_j$, for $j\in \mathbb{N}$.
\end{enumerate}
\end{lemma}
\begin{proof}
Let $v:=q-p$ and $\lambda(z):=z^TJv$. We consider two different cases. At first, we assume that
\begin{equation}\label{equation:discrete}
\{\lambda(p)\}\cup \{\lambda(a_j)\}_{j=1}^m\cup \{\lambda(c_j)\}_{j\in \mathbb{N}}
\end{equation}
is a discrete sequence without repetition. An application of Mittag-Leffler's osculation theorem implies the existence of a holomorphic function $f \colon \C\to\C$ with the properties
\begin{enumerate} [label=\arabic*)]
    \item $f(\lambda(p))=1$
    \item $f(\lambda(z))=O(|z-a_j|^N),\, z \to a_j, \, j=1,\dots,m$
    \item $\sup_{z\in \lambda(K)}|f(z)|< \epsilon /|v| $
    \item $f(\lambda(c_j))=0$
\end{enumerate}
and then the mapping $F \colon \C^{2n}\to \C^{2n}$ given by
\[ F(z):=z + f(\lambda(z))v,\]
is the desired symplectic automorphism.
This proves the lemma in the case when (\ref{equation:discrete}) is satisfied.
 
Now consider the case where
\begin{align*}
\{\lambda(p)\}\cup \{\lambda(a_j)\}_{j=1}^m\cup \{\lambda(c_j)\}_{j\in \mathbb{N}}
\end{align*}
is not a discrete set without repetition. Equivalently, since this set is actually discrete, the sequence \[\{{p,a_1,\dots,a_m,c_1,c_2,\dots\}}\] is not mapped injectively to $\C$ by $\lambda_v$. This is exactly then the case if $\lambda(\Delta)=\Delta^tJv=0$ or $\lambda(p)=p^tJq\in \{\lambda(a_j)\}_{j=1}^m\cup \{\lambda(c_j)\}_{j\in \mathbb{N}}$. For an arbitrary point $r\in \C^{2n}$ we write $v_1:=r-p$ and $v_2:=q-r$. We can choose $r$ such that
\[\{\lambda_{v_i}(p)\}\cup \{\lambda_{v_i}(a_j)\}_{j=1}^m\cup \{\lambda_{v_i}(c_j)\}_{j\in \mathbb{N}},\quad i=1,2,\]
are discrete sequences without repetition. To see this, we choose $r$ such that
\begin{align}\label{condition1}
v_i^TJ\Delta \neq 0\quad i=1,2,
\end{align}
which implies that $E_i:=\{\lambda_{v_i}(a_j)\}_{j=1}^m\cup\{\lambda_{v_i}(c_j)\}_{j\in \mathbb{N}}$ is discrete without repetition. If necessary, we perturb $r$ slightly such that (\ref{condition1}) remains valid but such that $\lambda_{v_i}(p)\in\C\setminus E_i$.
We now apply the lemma to obtain a symplectic automorphism $F_1$ satisfying properties (ii), (iii) with $\epsilon/2$, (iv) and $F_1(p)=r$. Let $K'$ denote the closure of the $\frac{\epsilon}{2}$-neighborhood of $K$. Then $F_1(K)\subset K'$. We apply the lemma once more, this time with $K'$ instead of $K$, to obtain a symplectic automorphism $F_2$ satisfying properties (ii), (iii) with $\epsilon/2$, (iv) and $F_2(r)=q$. Then the composition $F:=F_1\circ F_2$ is the desired symplectic automorphism and this finishes the proof. 
\end{proof}

%
%
\section{Interpolation at infinitely many points: Proof of Theorem \ref{theorem1} }
As both sequences $a_j$ and $b_j$ are symplectically tame, we can adjust the base points of the family of jets and assume that $a_j=b_j=(j,\dots,j)=j\Delta$, $j\in \mathbb{N}$. Furthermore, we only need to prove this result at a discrete sequence $\{c_j\}_{j\in \mathbb{N}}$ of points contained in $\SpanD$, as for any such sequence there exists a symplectic automorphism $\Phi\in \Symplecto$ such that
\[\Phi(z) = j\Delta + (z-c_j) + O(|z-c_j|^{m_j+1}),\quad z\to c_j,\quad j\in \mathbb{N},\]
by Proposition \ref{prop:symptame}.

Fix an exhausting sequence of compacts $T_1\subset T_2\subset \cdots \subset \cup_{j=1}^\infty T_j=X$ and a sequence of positive real numbers $\{\epsilon_j\}_{j\in \mathbb{N}}$, such that $\sum_{j=1}^\infty \epsilon_j <+\infty$. We will inductively construct the following:
\begin{enumerate}[label=(\alph*)]
\item a discrete sequence of points $\{\alpha_j\}_{j\in \mathbb{N}}\subset \mathbb{N}$,
\item an exhausting sequence of convex compacts $K_1\subset K_2\subset \cdots \subset \cup_{j=1}^\infty K_j=\C^{2n}$ such that $\mathrm{dist}(K_{j-1},\C^{2n}\setminus K_j) > \epsilon_j$ and $\alpha_j\Delta \not\in K_j$ for all $j>0$,
\item a sequence of holomorphic maps $\Psi_j \colon X\to\Symplecto$ for $j\in \mathbb{N}$, such that for $F_x^k:=\Psi_k^x\circ\cdots \circ \Psi_1^x\in \Symplecto$, $x\in X$, $k\in \mathbb{N}$, we have that
\begin{enumerate}
\item[($i_k$)] $F_x^k(z)=\tau^*(P_x^j(z))+O(|z-\alpha_j\Delta|^{m_j+1})$ for $z\to \alpha_j\Delta$ and each $j=1,\dots,k$, where $\tau$ is the translation mapping $j$ to $\alpha_j$ for all $j\in \mathbb{N}$.
\item[($ii_k$)] $F_x^k(i\Delta)=i\Delta$ for every $i> \alpha_k$,
\item[($iii_k$)] $\Psi_j^k$ is $\epsilon_j$-close to the identity on $K_j$ for every $x\in T_j$ and $j\in \mathbb{N}$.
\end{enumerate}
\end{enumerate}

For the base case of the induction, let $K_1=\mathbb{B}$ be the unit ball in $\C^{2n}$, and $\alpha_1=2$. By Proposition \ref{theorem:finiteInterpol} we can pick a family of symplectic automorphisms $\Psi_1^x\in \Symplecto$, $x\in X$, such that properties ($i_1$), ($ii_1$) and ($iii_1$) are satisfied.

For the induction step, suppose we have constructed the objects in (a), (b) and (c) satisfying properties ($i_j$), ($ii_j$) and ($iii_j$) for $j=1,\dots,k$. Pick a compact convex set $K_{k+1}\subset \C^{2n}$ such that 
\begin{align}\label{condition1b}
(\alpha_k+1)\mathbb{B}\cup F_x^k((\alpha_k+1)\mathbb{B})\subset K_{k+1},\quad x\in T_{k+1},
\end{align}
 and \[\mathrm{dist}(K_k,\C^{2n}\setminus K_{k+1}) > \epsilon_{k+1}.\]
Choose $\alpha_{k+1} \in \mathbb{N}$ such that $\alpha_{k+1}\Delta \not\in K_{k+1}$. We again invoke Proposition \ref{theorem:finiteInterpol} to obtain a holomorphic map $\Psi_{k+1} \colon X\to\Symplecto$ with the following properties:
\begin{enumerate}[label=\arabic*)]
    \item $\Psi_{k+1}^x(z)= z + O(|z-\alpha_j\Delta|^N)$, $z\to \alpha_j\Delta$ for every $j=1,\dots,k$, where the integer $N>m_j$ for every $j<k+1$,
    \item $\Psi_{k+1}^x(z)= P_x^{k+1}\circ (F_k^x)^{-1}(z) + O(|z-\alpha_{k+1}\Delta|^{m_{k+1}+1})$ as $z\to \alpha_{k+1}\Delta$
    \item $\Psi_{k+1}^x$ is $\epsilon_{k+1}$-close to the identity on $K_{k+1}$ for every $x\in T_{k+1}$
    \item $\Psi_{k+1}^x(j\Delta)=j\Delta$ for every $j>\alpha_{k+1}$.
\end{enumerate}

We then see that the holomorphic family of symplectic automorphisms defined by 
\[F_x^{k+1}=\Psi_{k+1}^x\circ F_x^k\in \Symplecto,\quad x\in X,\]
satisfies properties ($i_{k+1}$), ($ii_{k+1}$) and ($iii_{k+1}$), so the induction may proceed.

The sequence of compacts $K_j\subset \C^{2n}$, $j\in \mathbb{N}$, constructed in this way clearly satisfies condition (b). By (\ref{condition1b}) we can apply \cite{KutRamPeon}*{Lemma 4.1}, which yields that the sequence $\{F^k\}_{k\in \mathbb{N}}$ converges to a holomorphic family of automorphisms $F \colon X\to \mathrm{Aut}(\C^{2n})$ which interpolates the given families of jets $\tau^*(P_x^j)$ at the points $\alpha_j\Delta$, $j\in \mathbb{N}$, thanks to property (i). Moreover, we have
\[F^*\omega = \lim_{k\to\infty} (F^k)^*\omega = \omega,\]
which shows $F\in \Symplecto$. This finishes the proof of Theorem \ref{theorem1}.
\newpage

\section{Symplectically tame sets}
In this section, we prove some extensions of the results of Rosay and Rudin for tame sets to the symplectic holomorphic setting. In Subsection \ref{subsec-unavoidable} we construct an unavoidable set for symplectic holomorphic maps fixing the origin.

\smallskip

The following well-known generalization of Mittag-Leffler's theorem can be proved using the coherence of the powers of the ideal sheaf $\mathcal{I}_A$ for an analytic subset $A$ of a Stein space $X$ and Cartan's Theorem B.
\begin{prop}\label{sympl:question}
Given a discrete sequence $\{w_k\}_{k\in \mathbb{N}}\subset \C^n$ without repetition and $\{z_k\}_{k\in \mathbb{N}}\subset \C^n$, there exists a holomorphic function $f \colon \C^n\to \C$ with $\nabla f(w_k)=z_k.$
\end{prop}

A linear subspace $L \subset \CC{2n}$ is called \textit{Lagrangian}, if $L=L^{\bot}$, where
\[ L^{\bot} = \{z\in \CC{2n}: \omega(z,w)=0,\ \forall w\in L\}\]
is the \textit{symplectic complement} of $L$.
\begin{theorem} \label{langrage}
Let $E=\{c_1,c_2,\dots\} \subset \C^{2n}$ be a discrete sequence, $L$ a Lagrangian subspace, and $\pi$ the projection $\C^{2n} \to L$.
\begin{enumerate}
\item[(a)] If $\pi$ is injective and the image $\pi(E)$ is discrete, then $E$ is symplectically tame.
\item[(b)] If $\pi$ has finite fibers and the image $\pi(E)$ is discrete, then $E$ is symplectically tame.
\end{enumerate}
\end{theorem}
\begin{proof}
We start proving (a). Without loss of generality, we may choose the standard Lagrangian subspace $L =\{0\}^n\times \CC{n}$. Write $E=\{c_1,c_2,\dots\}$ and $c_k=(z_k,w_k)\in \CC{n}\times \CC{n}$. By assumption, the sequence $\{w_k\}_{k\in \mathbb{N}}$ is discrete and without repetition. By Proposition \ref{sympl:question}, there exists $f \colon \CC{n} \to \C$ holomorphic with $\nabla f(w_k)=-z_k+ke_1$. Then
\[ \Psi_1(z,w) = (z + \nabla f(w), w)\]
defines a symplectic automorphism which maps $E$ to the discrete set 
\[  E':=\{(e_1,w_1),(2e_1,w_2),\dots,(ke_1,w_k),\dots\}
\]
Similarly, we find a holomorphic map $g \colon \CC{n}\to \C$ with $\nabla g(ke_1) = -w_k$. And hence the symplectic automorphism $\Psi_2(z,w) = (z, w+\nabla g(z))$ maps $E'$ to the set $\mathbb{N}\cdot e_1$.

For (b), we want to find a symplectic automorphism $\Psi$ such that the projection of $\Psi(E)$ onto $L$ is injective. Then we apply (a).

For each $k\in \mathbb{N}$, there exist $z_{k,1},\dots,z_{k,n_k}\in \CC{n}$ such that $\{(z_{k,1},w_k), \dots, (z_{k,n_k},w_k)\}$ is the fiber of $w_k$. Inductively, we can choose sequences $\{R_k\}_{k\in \mathbb{N}} \subset \mathbb{R}, \{b_k\}_{k\in \mathbb{N}} \subset \CC{n}$ such that
\[ R_{k+1} > |z_{k,j} + b_k| > R_k\]
for all $k\geq 1$ and $j=1,\dots,n_k$. As before, there exists a holomorphic map $f \colon \CC{n}\to \C$ such that $\nabla f(w_k)=b_k$. The symplectic automorphism
\[ \Psi(z,w) = (z+\nabla f(w),w)\]
satisfies the desired properties.
\end{proof}
\begin{Rem}
    The assumptions of statement (b) can be further relaxed to that  all fibers of $\pi$ are finite except over finitely many points. 
\end{Rem}

\begin{cor}
    Let $E\subset \C^{2n}$ be a symplectically tame set and $A\subset \C^{2n}$ a finite set of points. Then $E\cup A$ is symplectically tame.
\end{cor}
\begin{proof}
    Without loss of generality we may assume that $E$ is the standard tame set $\mathbb{N}\cdot \Delta = \{(n,\dots, n)\in \C^{2n}: n\in \mathbb{N}\}.$ Projection to a Lagrangian is discrete with finite fibers. This remains valid if we add a finite number of points to $E$.
\end{proof}

\begin{cor}
    Every discrete infinite set $E$ in $\C^{2n}$ is either symplectically tame or the union of two symplectically tame sets. 
\end{cor}
\begin{proof}
    We divide $E$ into two disjoint subsets 
    \[
        E_1 = \{ z + w \in E : \lvert z \rvert \ge \lvert w \rvert \}, \quad
        E_2 = \{ z + w \in E : \lvert z \rvert < \lvert w \rvert \}
    \]
    Assume that both are infinite. Since there are at most finitely many points of $E_1$ over every compact set in $\C^n_z$, the projection $\pi_z(E_1) \subset \C^n$ is discrete and $E_1$ is symplectically tame by Theorem \ref{langrage}. Similarly $E_2$ is symplectically tame. If $E_1$ is finite, then its union with the symplectically tame $E_2$ is symplectically tame.  
\end{proof}

Let the map $\pi_j(z)=(z_j,z_{n+j})$ be the projection of $\C^{2n}$ onto the $j$-th symplectic plane.

\begin{theorem}
Let $E\subset \C^{2n}$ be a discrete set and suppose that the projections $\pi_j(E)$, $j=1,\dots,n$, are very tame in $\C^2$. Then $E$ is symplectically tame.
\end{theorem}

\begin{proof}
If $\pi_j(E)$ is very tame, there exists a volume-preserving holomorphic automorphism $\Phi=(\Phi_1,\Phi_2)\in \mathrm{Aut}_1(\C^2)$ such that $\Phi(\pi_j(E))=\mathbb{N}\cdot e_1$. The map
\begin{align*}
&\Psi_j(z_1,\dots,z_{2n}) = \\ 
& (z_1,\dots,z_{j-1},\Phi_1(z_j,z_{n+j}),z_{j+1},\dots,z_{n+j-1},\Phi_2(z_j,z_{n+j}),z_{n+j+1},\dots,z_{2n})
\end{align*}
is a symplectic automorphism of $\C^{2n}$. This argument works for all $j=1,\dots,n$ and the composition $\Psi_1\circ \cdots \circ \Psi_n$ is then a symplectic automorphism which maps $E$ to $\mathbb{N}^n\times \{0\}^n$. And this set is symplectically tame (see \cite{AndristUgolini}*{Lemma 3.5}).
\end{proof}

\begin{cor}
Let $E_1,\dots,E_n\subset \C^2$ be very tame sets. Then $E_1\times \cdots \times E_n$ is symplectically tame (with respect to the symplectic form $\omega' = dz_1\wedge dz_2 + \cdots + dz_{2n-1}\wedge dz_{2n}$).

In particular, $\mathbb{N}^{2n}\subset \C^{2n}$ is symplectically tame.
\end{cor}

\begin{proof}
The set $\mathbb{N}^2\subset \C^2$ is very tame by \cite{RosayRudin}*{Theorem 3.8}.
\end{proof}

We end with a  most natural question, which we unfortunately have not been able to solve yet. 
\begin{question}
Does there exist a very tame set in $\CC{2n}$ which is not symplectically tame? 
\end{question}

\subsection{An example of an unavoidable set by symplectic maps}
\label{subsec-unavoidable}
We construct an example of a discrete set which is unavoidable by holomorphic symplectic maps fixing the origin. The structure of the construction is very similar to that of Rosay and Rudin \cite{RosayRudin}. We therefore first generalize two results from their paper. The third ingredient is an application of the holomorphic Non-Squeezing Theorem (see \cite{Gaussier}*{Theorem 1.1}). This application allows a direct proof, in contrast to the proof by contradiction in \cite{RosayRudin}. 
\begin{lemma}[\cite{RosayRudin}*{Lemma 6.1}]
    Let $A \colon \C^n\to \C^n$ be a linear map with $\det(A)=1$ and $P:\C^n\to \C^n$ a linear projection with $\mathrm{rank}(P)=k$ for $k\in \{1,\dots,n-1\}$. Assume that $u\in \ker P$ with $|u|=1$. Then
    \[ |A^{-1}u| \leq \Norm{PA}^k,\]
where $\Norm{\cdot}$ denotes the usual operator norm, relative to the Euclidean metric on $\C^n$.
\end{lemma}
\begin{proof}
    We have $\C^n=\ker P \oplus \img P$. Set $U:=A^{-1}\ker P$ and $V:=A^{-1}\img P$. Then $\C^n=U\oplus V$. We choose
    \[w_1:=\frac{1}{|A^{-1}u|}A^{-1}u\in U\]
    and augment to an orthogonal basis $w_1,\dots,w_{n-k}$ of $U$ with $|Aw_j|=1$. Now choose an orthonormal basis $w_{n-k+1},\dots,w_n$ of $V$. Then $w_1,\dots,w_n$ is a basis of $\C^n$ and we use it to identify linear maps of $\C^n$ with matrices. Set $\lambda=|A^{-1}u|$. If $D$ is diagonal, with entries $(\lambda,1,\dots,1)$ on the main diagonal, then $ADw_1=u$, so that the columns of $AD$ are
    \[u,Aw_2,\dots, Aw_n.\]
    Note that $u,Aw_2,\dots, Aw_{n-k}$ forms a basis of $\ker P$. Since $P^2=P$, each vector $Aw_j-PAw_j$ lies in the kernel of $P$, hence is a linear combination of $u,Aw_2,\dots Aw_{n-k}$. The columns $Aw_{n-k+1},\dots Aw_n$ can therefore be replaced by $PAw_{n-k+1},\dots,PAw_n$, without changing the determinant of $AD.$ It follows from Hadamard's inequality that
    \begin{align*}
        |A^{-1}u| &= \lambda = \det (D) = \det(AD) \\
        &= \det\big( u,Aw_2,\dots,Aw_{n-k},PAw_{n-k+1},\dots, PAw_n\big) \\
        &\leq |u| \cdot |Aw_2| \cdots |Aw_{n-k}|\cdot |PAw_{n-k+1}| \cdots |PAw_n| \leq \Norm{PA}^k
    \end{align*}
    because $|u|=|Aw_2|=\cdots = |Aw_{n-k}|=|w_{n-k+1}|=\cdots = |w_n|=1$.
\end{proof}

We will use the notations $\mathbb{B}_k$ for the open unit ball in $\C^k$ and $\mathbb{B}_k(x)$ for the open ball in $\C^k$ with center $x$ and radius 1.
\begin{lemma}\cite{RosayRudin}*{Lemma 6.2}\label{lemma:delta}
    Given $0<a_1<a_2$, $r>0$, $k\in \{1,\dots,n-1\}$ there exists $\delta>0$, namely
    \[\delta = \bigg(\frac{k}{r}\bigg)^k\bigg(\frac{a_2-a_1}{k+1}\bigg)^{k+1}\]
    with the following property:

    If $F \colon a_2\mathbb{B}_n\to (r\mathbb{B}_k)\times \C^{n-k}$ is holomorphic, with $JF\equiv 1$, then $f(a_2\mathbb{B}_n)$ contains the ball
    \begin{equation}
        \{F(z)+ \sum_{j=k+1}^n \lambda_je_j: \sqrt{|\lambda_{k+1}|^2+\cdots + |\lambda_n|^2}<\delta\}
    \end{equation}
    for every $z\in a_1\mathbb{B}_n$.
\end{lemma}
\begin{proof}
    We consider the orthogonal projection $P$ in $\C^n$ whose kernel is spanned by $e_{k+1},\dots,e_n$. For every $u\in \ker P$ with $|u|=1$ we apply the same argument as in the proof of Lemma 6.2 in \cite{RosayRudin}.
\end{proof}
The following lemma is an application of the Non-Squeezing Theorem. We consider symplectic maps $f:\C^{2n}\to \C^{2n}$ with respect to the symplectic form
\[\omega' = dz_1\wedge dz_2 + \cdots + dz_{2n-1}\wedge dz_{2n}.\]
\begin{lemma}\label{lemma:applicationNonSqueezing}
    Let $a_1>1$, $a_{k+1}:=a_k+\frac{1}{k^2}$ for $k\geq 1$, and $a:=\lim_{k\to \infty} a_k$. Let $f \colon \C^{2n}\to\C^{2n}$ be a holomorphic symplectic map with $f(0)=0$. Then there exists a natural number $N$ such that
    \begin{equation}\label{equation:subset}
        f(a_k\mathbb{B}_{2n})\subset k\mathbb{B}_2\times \C^{2n-2}
    \end{equation}
    for every $k>N$ but $f(a_N\mathbb{B}_{2n})\not\subset N\mathbb{B}_2\times \C^{2n-2}$.
\end{lemma}
\begin{proof}
    Note that $a=a_1+\frac{\pi^2}{6}<\infty$. By continuity, there exists $R>0$ such that $f(a\mathbb{B}_{2n})\subset R\mathbb{B}_2\times \C^{2n-2}$. Since $a_k<a$ for all $k$, we have
    \[ f(a_k\mathbb{B}_{2n})\subset f(a\mathbb{B}_{2n})\subset R\mathbb{B}_2\times \C^{2n-2}\subset k\mathbb{B}_2\times \C^{2n-2}\]
    for every $k>R$. On the other hand, since $a_1>1$, the Non-Squeezing Theorem \cite{Gaussier} implies
    \[ f(a_1\mathbb{B}_{2n})\not\subset \mathbb{B}_2\times \C^{2n-2}\]
    hence there exists $k$ not satisfying (\ref{equation:subset}). Let $N<R$ be the largest such number.
\end{proof}
 Let $\Gamma$ denote the class of holomorphic symplectic maps $f \colon \C^{2n}\to \C^{2n}$ with $f(0)=0$.
\begin{lemma}
    Symplectically tame sets are avoidable by $\Gamma$.
\end{lemma}
\begin{proof}
    Let $E\subset \C^{2n}$ be symplectically tame. Without loss of generality, we may assume $E=\mathbb{N}\cdot e_1$. This set can be avoided by the symplectic map 
    \[f(z_1,\dots, z_{2n}) = (z_1e^{-z_2},e^{z_2},z_3,\dots,z_{2n}).\]
    This map does not belong to $\Gamma$, since $f(0)\neq 0$. However, there exists a symplectic automorphism $G\in \Symplecto$ with $G(N)=N$ and $G(f(0))=0$. Then $G\circ f \in \Gamma$ and it avoids $E$.
\end{proof}
Let $\pi'(z_1,\dots,z_{2n})=(z_1,z_2)$ and $\pi''(z_1,\dots,z_{2n})=(z_3,\dots,z_{2n})$.  For a holomorphic map $f \colon \C^{2n}\to\C^{2n}$ we write $f_{\pi'}:=\pi'\circ f$ and $f_{\pi''}:=\pi''\circ f$.
\begin{Exa}  
\label{theorem:unavoidable}
    There exists a tame set $E\subset \C^{2n}$ which is $\Gamma$-unavoidable.
\end{Exa}
\begin{proof}
    We first construct the set $E$ and then prove that it is unavoidable by maps in $\Gamma$.

    For every natural number $j$ we construct the following:\\
    Choose a finite set $E'_j:=E(a_j,a_{j+1},j,j+1,1)\subset \partial(j\mathbb{B}_2)$ in terms of Lemma 4.3 in \cite{RosayRudin}, that is, if $f_{\pi'} \colon a_{j+1}\mathbb{B}_{2n}\to (j+1)\mathbb{B}_2$ is holomorphic, with $|f_{\pi'}(0)|\leq \frac{1}{2}j$,
    \[\bigg\vert \frac{\partial(f_1,f_2)}{\partial(z_1,z_2)}\bigg\vert \geq 1,\quad \text{at some point of } a_j\overline{\mathbb{B}_{2n}}\]
    and if $f_{\pi'}(a_j\mathbb{B}_{2n})$ intersects $\partial(j\mathbb{B}_2)$, then $f_{\pi'}(a_{j+1}\mathbb{B}_{2n})$ intersects $E_j'.$

    Let $\delta_j>0$ be as in Lemma \ref{lemma:delta} with $a_{j+1},a_{j+2}, k=2$ and $r=j+2$, that is,
    \[ \delta_j = \bigg(\frac{2}{j+2}\bigg)^2\bigg(\frac{a_{j+2}-a_{j+1}}{3}\bigg)^{3}.\]
    Choose a discrete set $E_j''$ in $\C^{2n-2}$ such that each open ball with radius $\delta_j$ contains an element of $E_j''$. Then set
    \[E_j:=E_j'\times E_j'' \subset \partial(j\mathbb{B}_2)\times \C^{2n-2}\]
    and finally \[E:=\bigcup_{j=1}^{\infty} E_j.\]
 Observe that $\pi'(E)$ is discrete and therefore $E$ is tame.

    In order to prove that $E$ is unavoidable, we consider an arbitrary holomorphic symplectic map $f \colon \C^{2n}\to\C^{2n}$ with $f(0)=0$. Without loss of generality we may assume that $f'(0)=I_{2n}$, because otherwise we consider $g:=f\circ f'(0)^{-1}$ which is a holomorphic symplectic map with $g(0)=0$, $g'(0)=I_{2n}$ and $\img(f)=\img(g).$
    By Lemma \ref{lemma:applicationNonSqueezing}, there exists a natural number $N$ such that
    \[f_{\pi'}(a_k\mathbb{B}_{2n})\subset k\mathbb{B}_2\]
    for all $k>N$, but $f_{\pi'}(a_N\mathbb{B}_{2n})$ intersects $\partial(N\mathbb{B}_{2})$. This means that \[f_{\pi'}\vert_{a_{N+1}\mathbb{B}_{2n}} \colon a_{N+1}\mathbb{B}_{2n}\to (N+1)\mathbb{B}_2\] is a holomorphic map satisfying the assumptions of Lemma 4.3 in \cite{RosayRudin} so that $f_{\pi'}(a_{N+1}\mathbb{B}_{2n})$ intersects $E_N'$. Choose $z\in a_{N+1}\mathbb{B}_{2n}$ such that $f_{\pi'}(z)\in E_N'$. By Lemma \ref{lemma:delta}, $f(a_{N+2}\mathbb{B}_{2n})$ contains the ball
    \[ \{f_{\pi'}(z)\}\times \delta_N \mathbb{B}_{2n-2}(f_{\pi''}(z))\]
    and since $\delta_N \mathbb{B}_{2n-2}(f_{\pi''}(z))$ contains an element of $E_N''$ by construction, $f(a_{N+2}\mathbb{B}_{2n})$ contains an element of $E_N\subset E$.
\end{proof}

\begin{Rem}
    The set $\pi'(E)$ is discrete, but not very tame. Suppose it was very tame and $\phi\in \mathrm{Aut}_1(\C^2)$ the automorphism with $\phi(\pi'(E))=\{0\}\times \mathbb{N}$. Then
    \[\psi(z,w)=(\phi(z),w)\in \Symplecto\]
    is a symplectic automorphism sending $E$ into the hyperplane $H:=\{0\}\times \C^{2n-1}$. Since $H\setminus\{0\}$ is avoidable by $\Gamma$, this leads to a contradiction.
\end{Rem}
\begin{question}
Is $E$ avoidable by holomorphic maps $f \colon \C^{2n}\to \C^{2n}$ with $f(0)=0$ and $Jf\equiv const.$? Is $E$ even very tame?
\end{question}

\section*{Funding}
The author was supported by the European Union (ERC Advanced grant HPDR, 101053085 to Franc Forstneri\v{c}) and grant N1-0237 from ARRS, Republic of Slovenia. The second, third and fourth author were partially supported by Schweizerische Nationalfonds Grant 00021-207335.

\section*{Conflict of Interest}
The authors have no relevant competing interest to disclose.

%
%


\begin{thebibliography}{99}

\bib{AndristUgolini}{article}{
   author={Andrist, Rafael B.},
   author={Ugolini, Riccardo},
   title={A new notion of tameness},
   journal={J. Math. Anal. Appl.},
   volume={472},
   date={2019},
   number={1},
   pages={196--215},
   issn={0022-247X},
   review={\MR{3906368}},
   doi={10.1016/j.jmaa.2018.11.018},
}

\bib{MR1758584}{article}{
   author={Buzzard, Gregery T.},
   author={Forstneric, Franc},
   title={An interpolation theorem for holomorphic automorphisms of ${\bf
   C}^n$},
   journal={J. Geom. Anal.},
   volume={10},
   date={2000},
   number={1},
   pages={101--108},
   issn={1050-6926},
   review={\MR{1758584}},
   doi={10.1007/BF02921807},
}

\bib{MR1760722}{article}{
   author={Forstneric, Franc},
   title={Interpolation by holomorphic automorphisms and embeddings in ${\bf
   C}^n$},
   journal={J. Geom. Anal.},
   volume={9},
   date={1999},
   number={1},
   pages={93--117},
   issn={1050-6926},
   review={\MR{1760722}},
   doi={10.1007/BF02923090},
}

\bib{Forstneric:Actions}{article}{
   author={Forstneric, Franc},
   title={Actions of $( \mathbb{R},+)$ and $( \mathbb{C},+)$ on complex manifolds},
   journal={Math. Z.},
   volume={223},
   date={1996},
   number={1},
   pages={123--153},
   issn={0025-5874},
   review={\MR{1408866}},
   doi={10.1007/PL00004552},
}


\bib{FoK:thefirst}{article}{
   author={Forstneri\v{c}, F.},
   author={Kutzschebauch, F.},
   title={The first thirty years of Anders\'{e}n-Lempert theory},
   journal={Anal. Math.},
   volume={48},
   date={2022},
   number={2},
   pages={489--544},
   issn={0133-3852},
   review={\MR{4440754}},
   doi={10.1007/s10476-022-0130-1},
}








\bib{Gaussier}{article}{
   author={Gaussier, Herv\'e},
   author={Joo, Jae-Cheon},
   title={Non-squeezing property for holomorphic symplectic structures},
   journal={Math. Res. Lett.},
   volume={22},
   date={2015},
   number={3},
   pages={729--740},
   issn={1073-2780},
   review={\MR{3350102}},
   doi={10.4310/MRL.2015.v22.n3.a6},
}

\bib{KutRamPeon}{article}{
   author={Kutzschebauch, Frank},
   author={Ramos-Peon, Alexandre},
   title={An Oka principle for a parametric infinite transitivity property},
   journal={J. Geom. Anal.},
   volume={27},
   date={2017},
   number={3},
   pages={2018--2043},
   issn={1050-6926},
   review={\MR{3667419}},
   doi={10.1007/s12220-016-9749-0},
}


\bib{Lin}{article}{
   author={Lin, Zhangli},
   author={Zhou, Xiangyu},
   title={Various subgroups of ${\rm Aut}(\mathbb{C}^n)$},
   journal={Proc. Lond. Math. Soc. (3)},
   volume={126},
   date={2023},
   number={3},
   pages={880--899},
   issn={0024-6115},
   review={\MR{4563862}},
}

\bib{LowEtAl}{article}{
   author={L\o w, Erik},
   author={Pereira, Jorge V.},
   author={Peters, Han},
   author={Wold, Erlend F.},
   title={Polynomial completion of symplectic jets and surfaces containing
   involutive lines},
   journal={Math. Ann.},
   volume={364},
   date={2016},
   number={1-2},
   pages={519--538},
   issn={0025-5831},
   review={\MR{3451396}},
   doi={10.1007/s00208-015-1217-9},
}

\bib{RosayRudin}{article}{
   author={Rosay, Jean-Pierre},
   author={Rudin, Walter},
   title={Holomorphic maps from ${\bf C}^n$ to ${\bf C}^n$},
   journal={Trans. Amer. Math. Soc.},
   volume={310},
   date={1988},
   number={1},
   pages={47--86},
   issn={0002-9947},
   review={\MR{0929658}},
   doi={10.2307/2001110},
}

\bib{Schott}{article}{
    title={Holomorphic Factorization of Mappings into the Symplectic Group}, 
      author={Josua Schott},
      journal ={to appear in J. of Europ. Math. Soc. (accepted Feb 2024)}
}

\bib{UgoJetInterpolation}{article}{
   author={Ugolini, Riccardo},
   title={A parametric jet-interpolation theorem for holomorphic
   automorphisms of $\mathbb{C}^n$},
   journal={J. Geom. Anal.},
   volume={27},
   date={2017},
   number={4},
   pages={2684--2699},
   issn={1050-6926},
   review={\MR{3707990}},
   doi={10.1007/s12220-017-9778-3},
}





\end{thebibliography}
\end{document}